\documentclass[12pt]{article}

\usepackage{amssymb}
\usepackage{amsmath}
\usepackage{amsthm}

\numberwithin{equation}{section}
\theoremstyle{plain}
\newtheorem{lem}{Lemma}[section]
\newtheorem{df}{Definition}[section]
\newtheorem{tw}{Theorem}[section]
\newtheorem{ex}{Example}[section]
\newtheorem{rem}{Remark}[section]
\newcommand{\esssuptT}{\textrm{ess} \sup_{\hspace{-15pt}t\in T}}
\newcommand{\R}{\mathbb{R}}

\begin{document}

\title{{\bf A strong ergodic theorem  for   extreme and intermediate order statistics}}

\author{
{\bf Aneta 
Buraczy\'{n}ska}\\
Faculty of Mathematics and Information Science,\\
Warsaw University of Technology, \\
ul. Koszykowa 75, 00-662 Warsaw, Poland, \\
e-mail: a.buraczynska@mini.pw.edu.pl
\and
{\bf Anna
Dembi\'{n}ska}\thanks{Corresponding author}\\
Faculty of Mathematics and Information Science,\\
Warsaw University of Technology, \\
ul. Koszykowa 75, 00-662 Warsaw, Poland, \\
e-mail: dembinsk@mini.pw.edu.pl}

\maketitle

\begin{abstract}
We study almost sure limiting behavior of extreme and intermediate order statistics arising from strictly stationary sequences. First, we provide sufficient dependence conditions under which these order statistics converges almost surely to the left or right  endpoint of the population support, as in the classical setup of sequences of independent and identically distributed random variables. Next, we derive a~generalization of this result valid in the class of all strictly stationary sequences. For this purpose, we introduce notions of conditional left and right endpoints of the support of a random variable given a sigma-field, and present basic properties of these concepts. Using these new notions, we prove that extreme and intermediate order statistics from any discrete-time, strictly stationary process converges almost surely to some random variable. We discribe the distribution of the limiting variate. Thus we establish a strong ergodic theorem for extreme and intermediate order statistics.
\end{abstract}


\section{Introduction}
\label{sec1}
Let $(X_n, n\ge 1)$ be a~sequence of random variables (rv's) defined on the same probability space, and $X_{1:n}\le \cdots\le X_{n:n}$ be the order statistics corresponding to the sample $(X_1,\ldots,X_n)$. Following the standard notation, we will say that $(X_{k_n:n}, n\geq 1)$ is a sequence of (1) {\it extreme order statistics} if and only if (iff) $k_n$ or $n-k_n$ is fixed; (2)~{\it intermediate order statistics} iff $\min(k_n,n-k_n)\to\infty$ and $k_n/n\to\lambda\in\{0,1\}$ as $n\to\infty$; and (3) {\it central order statistics} iff $k_n/n\to\lambda\in(0,1)$  as $n\to\infty$.

In this paper we will focus on the asymptotic behavior of extreme and intermediate order statistics in the case when the sequence $(X_n, n\ge 1)$ forms a strictly stationary process. A lot is known about this behavior under some additional assumptions on the dependence structure between $X_i$'s. In particular, extreme value theory, dealing with limiting laws of suitably normalized extreme and intermediate order statistics, is well developed; see, for example, \cite{LLR83, LR88, MS00, F05, WRL82, Ch85} and the references given there. Yet there does not exist very much literature on the almost sure asymptotic behavior of  extreme and intermediate order statistics, even  in the case when $(X_n, n\ge 1)$ are independent and identically distributed (iid) rv's with common cumulative distribution function (cdf) $F$ satisfying some requirement. Under conditions that $F$ is sufficiently smooth, Watts \cite{W80} and Chanda \cite{Ch92} gave Bahadur-Kiefer-type representations for  intermediate order statistics of iid rv's. A~brief review on almost sure behavior of maxima of iid rv's can be found in Embrechts et al. \cite{embre1997}. Characterizations of the minmal almost sure growth of partial maxima of iid rv's, obtained among others by Klass \cite{K84, K85}, were generalized to extreme upper order statistics by Wang \cite{W97}.

In this paper, we concentrate on extension of the following almost sure property of  extreme and intermediate order statistics taken from  Embrechts et al. (\cite{embre1997},  Proposition 4.1.14).
\begin{tw}
\label{EKM}
If  $(X_n,n\ge1)$ is a~sequence of~iid rv's and $(k_n,n\ge 1)$ is a~sequence of  integers such that 
\begin{equation}
\label{warK}
1\leq k_n\le n\textrm{ for all }n\ge 1\textrm{ and } \lim_{n\to\infty}k_n/n=\lambda\in\{0,1\},
\end{equation}
then
\begin{equation}\label{zbieznosc}
X_{k_n:n}\xrightarrow{n\to\infty}\gamma_0^{X_1}   \; (\gamma_1^{X_1})\textrm{ almost surely according as }\lambda=0\ (\lambda=1),
\end{equation}
where $\gamma_0^{X_1}$ and $\gamma_1^{X_1}$ are the left and right endpoints of the support of $X_1$, respectively.
\end{tw}
Following Smirnov \cite{smirnov1952}, we can view the above theorem as an analog of the strong law of large numbers for  extreme and intermediate order statistics. Our aim is to give its extension to the class of strictly stationary processes. We present such an extension in the whole generality. Firstly, our main result, Theorem \ref{ergodyczneNowe}, holds in the class of all strictly stationary processes -- no assumptions on dependence structure of the sequence $(X_n, n\ge 1)$ are needed. Secondly, no restrictions on the common univariate cdf $F$ of $(X_n, n\ge 1)$ are required.

The paper is organized as follows. In Section 2, we provide sufficient conditions on the structure of a strictly stationary sequence $(X_n, n\ge 1)$ ensuring that \eqref{warK} still implies \eqref{zbieznosc}. In Section 3, we introduce concepts of the conditional left and right endpoints of the support of an rv given a sigma-field. We also present a brief exposition of basic properties of these concepts. Next in Section 4, we use the new notions to formulate and prove the main result of the paper. Namely we show that extreme and intermediate order statistics arising from any strictly stationary sequence of rv's converge almost surely to some rv and we describe the distribution of the rv appearing in the limit. In Section 5, we give examples of application of the main result to some special cases of stationary processes. For readers' convenience in Appendix we recall the notion of essential supremum and its existence property that are  needed in one of our proof.

Throughout the paper we use the following notation.
Unless otherwise stated, the rv's $X_n$, $n\geq 1$, exist in a probability space $(\Omega, \mathcal F, \mathbb{P})$. $\mathbb R$,  $\mathbb Z$  and $\mathbb{N}$ represent the sets of real numbers, integers and positive  integers, respectively. For an rv $X$ with cdf $F$ we set
$$\gamma_0^X:=\inf\{x\in\mathbb{R}:\ F(x)>0\} \; \textrm{ and } \; \gamma_1^X:=\sup\{x\in\mathbb{R}:\ F(x)<1\},$$ 
and we call  $\gamma_0^X$ ($\gamma_1^X$)  the left  (right) endpoint of the support of $X$.
We write $I(\cdot)$ for the indicator function, that is $I(x\in A)=1$ if $x\in A$ and $I(x\in A)=0$ otherwise.
By  $\xrightarrow{a.s}$  we denote  almost sure convergence and $a.s.$ stands for almost surely. Moreover, when in context  different probability measures appear, to avoid confusion, we write $\xrightarrow{\mathbb{P}-a.s}$ and $\mathbb{E}_{\mathbb{P}}$ for almost sure convergence and expectation with respect to the measure $\mathbb{P}$, respectively, and we say that an event $A$ is true $\mathbb{P}$-$a.s.$ if $\mathbb{P}(A)=1$. Finally, an extended rv in $(\Omega, \mathcal F, \mathbb{P})$ is a $\mathcal F$-measurable function $X: \Omega\to[-\infty,\infty]$. We assume the usual conventions about arithmetic operations in $[-\infty,\infty]$: if $a\in\R$ then $a\pm\infty=\pm\infty$, $a/\pm\infty=0$, $a\cdot\infty=\infty$ if $a>0$ and $a\cdot\infty=-\infty$ if $a<0$, $0\cdot(\pm\infty)=0$, $\infty+\infty=\infty$, $-\infty-\infty=-\infty$.


\section{Stationary and ergodic sequences}
\label{sec2}

The aim of this section is to relax the idd assumption in Theorem \ref{EKM}. More precisely, we will show that the conclusion of this theorem will remain unchanged if the condition that $(X_n,n\ge 1)$ is an~iid sequence is replaced by a~weaker one that $(X_n,n\ge 1)$ forms a~strictly stationary and ergodic process.

\begin{tw}\label{l3}
Let $\mathbb{X}=(X_n;n\ge 1)$ be a strictly stationary and ergodic sequence of rv's with any cdf~$F$ and let $(k_n;n\ge 1)$ be a~sequence of  integers satisfying~\eqref{warK}. Then~\eqref{zbieznosc} holds.
\end{tw}

\begin{proof}
We assume that $\lambda=0$ since the case $\lambda=1$ can be easily transformed to the former by considering $(-X_n,n\ge 1)$ instead of~$(X_n,n\ge 1)$.

First note that, for all $n\ge 1$,  $X_{k_n:n}\ge \gamma_0^{X_1}$ a.s. Therefore we are reduced to showing that 
$\limsup_{n\to\infty}X_{k_n:n}\le\gamma_0^{X_1}$ a.s.

Define $d_m=-m$ if $\gamma_0^{X_1}=-\infty$ and $d_m=\gamma_0^{X_1}+\frac{1}{m}$ otherwise. Fix $m\ge 1$. By the assumption that $\mathbb{X}$ is strictly stationary and ergodic, the sequence $(I(X_n\le d_m),n\ge 1)$ is so as well.
This is a simple consequence of Proposition 2.10 of Bradley~\cite{brad2007}. The classic strong ergodic theorem 
(see, for example, Grimmet and Stirzaker (\cite{grim2004}, Chapter 9.5) gives, as $n\to\infty$,
\begin{equation}\label{dtw1}
\sum_{i=1}^{n}I(X_i\le d_m)/n\xrightarrow{a.s.}\mathbb{E}(I(X_1\le d_m))=\mathbb{P}(X_1\le d_m)=F(d_m)>0.
\end{equation}
Since by assumption $k_n/n\to 0$, we get
\[
\frac{\sum_{i=1}^{n}I(X_i\le d_m)}{n}-\frac{k_n}{n}\xrightarrow{a.s.}F(d_m)>0\textrm{ as }n\to\infty
\]
and therefore
\begin{align*}
\mathbb{P}(X_{k_n:n}\le d_m\textrm{ for all large }n)=
\mathbb{P}\Big(\sum_{i=1}^{n}I(X_i\le d_m)\ge k_n\textrm{ for all large }n\Big)\\
=\mathbb{P}\Big(\frac{\sum_{i=1}^{n}I(X_i\le d_m)}{n}-\frac{k_n}{n}\ge 0\textrm{ for all large }n\Big)=1,
\end{align*}
which means that $\limsup_{n\to\infty}X_{k_n:n}\le d_m$ a.s. 
Letting $m\to\infty$ and using the countability of  $\mathbb{N}$ yields $\limsup_{n\to\infty}X_{k_n:n}\le\gamma_0^{X_1}$. This completes the proof.
\end{proof}
It is worth emphasizing that Theorem~\ref{l3} applies to all strictly stationary and ergodic sequences of rv's -- in particular no restriction is imposed on the cdf $F$ of $X_i$. The class of strictly stationary and ergodic processes is very broad. It includes, for example, the family of linear processes, that is processes defined by
\[
X_n=\sum_{i=-\infty}^{\infty}a_i\varepsilon_{n-i},\ n\ge 1,
\]
where $\varepsilon_k$, $k\in\mathbb{Z}$, are iid rv's and $a_i$, $i\in\mathbb{Z}$, are real coefficients such that $X_n$ exists almost surely. The family of linear processes covers, among others, all stationary autoregressive-moving average processes and all Gaussian processes with absolutely continuous spectrum.

\medskip
In the proof of Theorem~\ref{l3} we needed ergodicity of $\mathbb{X}$ only to show that this implies that of the sequences $(I(X_n\le d_m),n\ge 1)$, $m\geq 1$. This observation leads to the following result.
\begin{tw}
Theorem~\ref{l3} is still true if we replace the assumption that $\mathbb{X}=(X_n;n\ge 1)$ is ergodic by the condition that, for every $x$ belonging to the support of $X_1$,
\begin{equation}\label{autoC}
\sum_{i=1}^{\infty}c_x(i)/i<\infty,
\end{equation}
where $c_x(i)=\mathbb{P}(X_1\le x, X_{1+i}\le x)-\mathbb{P}(X_1\le x)^2$, $i\ge 1$, is the autocovariance function of the process $(I(X_n\le x),n\ge 1)$.
\end{tw}
\begin{proof}
Dembi\'nska~\cite{dem2012} showed that~\eqref{autoC} gives $\sum_{i=1}^{n}I(X_i\le x)/n\xrightarrow{a.s.}F(x)$ as $n\to\infty$. Thus we have~\eqref{dtw1}. The rest of the proof runs as before.
\end{proof}

We have shown that the assumption that $\mathbb{X}$ is ergodic can be replaced by another one and the conclusion of Theorem~\ref{l3} will remain unchanged. Yet, this assumption cannot be completely dropped as the following example shows.
\begin{ex}
\label{ex1}
Let X be some non-degenerate rv and $X_n=X$ for all $n\ge 1$. Then the sequence $(X_n,n\ge 1)$ is strictly stationary but
\[
X_{k_n:n}=X\xrightarrow{a.s.}X\neq\gamma_0^{X}.
\]
\end{ex}
We see that, if we assume only strictly stationarity of~$\mathbb{X}$, then the almost sure limit of $X_{k_n:n}$ need not to be a~constant -- it can be a~non-degenerate rv. The rest of the paper is devoted to the proof of almost sure existence of $\lim_{n\to\infty}X_{k_n:n}$ under the single assumption that $\mathbb{X}$ is strictly stationary and to the description of the distribution of the limiting rv.


\section{Conditional left and right endpoints of the support}
\label{sec3}

Tomkins \cite{tomkins} proposed a~definition of conditional median. This definition has been extended to other quantiles as follows.

\begin{df}
Suppose $X$ is an~rv on a~probability space $(\Omega,\mathcal{F},\mathbb{P})$, $\mathcal{G}\subseteq\mathcal{F}$ is a~sigma-field and $\lambda\in(0,1)$. Then an~rv~$Q_{\lambda}$ with the following properties
\begin{description}
    \item[(i)]
    $Q_{\lambda}$ is $\mathcal{G}$-measurable,
    \item[(ii)]
    $\mathbb{P}(X\ge Q_{\lambda}|\mathcal{G})\ge 1-\lambda$ a.s. and $\mathbb{P}(X\le Q_{\lambda}|\mathcal{G})\ge\lambda$ a.s.
\end{description}\
is called a~conditional $\lambda$th quantile of~$X$ with respect to~$\mathcal{G}$.
\end{df}
Using the concept of conditional quantile, Dembi\'nska \cite{dem2014} described the distribution of the rv appearing as the almost sure limit of central order statistics from stationary processes. The aim of the present paper is to give a~corresponding result for extreme and intermediate order statistics. To do this, we need an extension of the notion of conditional quantiles to the case of $\lambda=0$ and $\lambda=1$. This extension leads us to new concepts of conditional left and right endpoints of the support of an~rv. Before we introduce  these concepts, we first establish some properties that will guarantee the correctness of the proposed definitions.

\begin{lem}\label{l1}
Let $\mathcal{G}\subseteq\mathcal{F}$ be a sigma-field. Suppose $(X_n, n\geq 1)$ and $(Q_n,  n\geq~1)$ are sequences of rv's and extended rv's from probability space $(\Omega,\mathcal{F},\mathbb{P})$, respectively, such that
\begin{description}
\item[(i)]
$Q_n$ is $\mathcal{G}$-measurable and $\mathbb{P}(X_n\ge Q_n|\mathcal{G})=1$ a.s. for all $n\geq 1$, and
\item[(ii)]
there exist an rv~$X$ and an extended rv~$Q$ such that $X_n\xrightarrow{a.s.} X$  and $Q_n\xrightarrow{a.s.} Q$.
\end{description}
Then $\mathbb{P}(X\ge Q|\mathcal{G})=1$ a.s. 
\end{lem}

\begin{proof}
Let  $X_n\xrightarrow{a.s.} X$  and $Q_n\xrightarrow{a.s.} Q$. Then
\begin{equation}
\label{wzor1}
I(X\ge Q)\ge \limsup_{n\to\infty}I(X_n\ge Q_n) \; \textrm{ a.s.},
\end{equation}
because otherwise there would exist a set $\Omega_0\subseteq\Omega$  of positive probability such that 
\begin{equation}
\label{lem3.1wd1}
I(X(\omega)\ge Q(\omega))=0  \textrm{ and } \limsup_{n\to\infty}I(X_n(\omega)\ge Q_n(\omega))=1  \textrm{ for  } \omega\in\Omega_0.
\end{equation}
 If so, for  $\omega\in\Omega_0$, we would have  $X_n(\omega)\ge Q_n(\omega)$  for infinitely many $n$.  By assumption (ii) it would give $X(\omega)\ge Q(\omega)$  for almost all $\omega\in\Omega_0$. Therefore, in this case, $I(X(\omega)\ge Q(\omega))=1$  for almost all $\omega\in\Omega_0$, which shows that (\ref{lem3.1wd1}) is impossible. Thus (\ref{wzor1}) is proved.

By $\eqref{wzor1}$ and Fatou's Lemma for conditional expectation, we get
\begin{align*}
\mathbb{E}(I(X\ge Q)|\mathcal{G})\ge \mathbb{E}(\limsup_{n\to\infty}I(X_n\ge Q_n)|\mathcal{G})\ge \limsup_{n\to\infty} \mathbb{E}(I(X_n\ge Q_n)|\mathcal{G})\textrm{ a.s.} 
\end{align*}
Hence
\begin{align*}
\mathbb{P}(X\ge Q|\mathcal{G})\ge
\limsup_{n\to\infty}\mathbb{P}(X_n\ge Q_n|\mathcal{G})=\limsup_{n\to\infty}1=1\textrm{ a.s.},
\end{align*}
and the lemma follows. 
\end{proof}

\begin{tw}\label{tw.1}
    For any rv $X$ and any sigma-field $\mathcal{G}\subseteq\mathcal{F}$ there exists an~extended rv $\overline{Q}_0$ having the following properties 
    \begin{description}
    \item[(i)]
    $\overline{Q}_0$ is $\mathcal{G}$-measurable,
    \item[(ii)]
    $\mathbb{P}(X\ge \overline{Q}_0|\mathcal{G})=1$ a.s.,
    \item[(iii)]
    for any $\mathcal{G}$-measurable extended rv~$Q_0$ such that $\mathbb{P}(X\ge Q_0|\mathcal{G})=1$ a.s., we have $Q_0\le \overline{Q}_0$ a.s.
    \end{description}
\end{tw}

\begin{proof}
By $\{Q_t,t\in T\}$ let us denote the set of~all $\mathcal{G}$-measurable extended rv's $Q_t$ satisfying $\mathbb{P}(X\ge Q_t|\mathcal{G})=1$ a.s. Note that $\{Q_t,t\in T\}$ is non-empty since $\gamma_0^{X}$ belongs to this set.

Define $\overline{Q}_0$ to be the essential supremum of $\{Q_t,t\in T\}$:
\begin{equation}
\label{wzor2}
\overline{Q}_0:=\esssuptT Q_t.
\end{equation}
It is known that $\textrm{ess} \sup_{t\in T} Q_t$ exists and 
\begin{equation}
\label{wzor3}
\esssuptT Q_t=\sup_{t\in C} Q_t,
\end{equation}
where $C$ is some countable subset of~$T$; see Theorem A.1 in the Appendix.

We will show that $\overline{Q}_0$ given by \eqref{wzor2} satisfies conditions (i)-(iii) of Theorem~\ref{tw.1}. Requirement (iii) is an immediate consequence of the definition of essential supremum so it suffices to prove (i) and (ii). To this end, let $Q_{t}$ and $Q_{s}$ be two extended rv's belonging to the set $\{Q_t,t\in T\}$. Then $M:=\max(Q_{t},Q_{s})$ also belongs to $\{Q_t,t\in T\}$. Indeed, it is obvious that $M$ is $\mathcal{G}$-measurable. Moreover
\[
\mathbb{P}(X\ge\max(Q_t,Q_s)|\mathcal{G})=\mathbb{P}(X\ge Q_t|\mathcal{G})=1\textrm{ a.s.}
\]
on the event $[Q_t\ge Q_s]$ and
\[
\mathbb{P}(X\ge\max(Q_t,Q_s)|\mathcal{G})=\mathbb{P}(X\ge Q_s|\mathcal{G})=1\textrm{ a.s.}
\]
on the event $[Q_s> Q_t]$. Therefore $\mathbb{P}(X\ge\max(Q_t,Q_s)|\mathcal{G})=1$ a.s.
Now let $R_n=\max_{k\le n}Q_{t_k}$, $n\ge 1$, where $\{t_1,t_2,\ldots\}=C$ and $C$ is a countable subset of $T$  satisfying \eqref{wzor3}.   By induction, for any $n\ge 1$, $R_n$ is $\mathcal{G}$-measurable and  $\mathbb{P}(X\ge R_n|\mathcal{G})=1$ a.s. Moreover $R_n\uparrow \overline{Q}_0$ a.s. Hence Lemma~\ref{l1} gives \linebreak  $\mathbb{P}(X\ge\overline{Q}_0|\mathcal{G})=1$ a.s. Relation  \eqref{wzor3} makes it obvious that $\overline{Q}_0$ is $\mathcal{G}$-measurable and the proof is complete.
\end{proof}

\begin{df}
Suppose $X$ is an rv and $\mathcal{G}$ is a~sigma-field with $\mathcal{G}\subseteq\mathcal{F}$. Then the extended rv $\overline{Q}_0$ from Theorem~\ref{tw.1} is called a~conditional left endpoint of the support of rv $X$ with respect to~$\mathcal{G}$ and will be denoted by~$\gamma_0(X|\mathcal{G})$.
\end{df}

Note that $\gamma_0(X|\mathcal{G})$ is not necessarily uniquely determined but any two versions of $\gamma_0(X|\mathcal{G})$ agree a.s. A~version of $\gamma_0(X|\mathcal{G})$ can be also viewed as a~conditional quantile of order $\lambda=0$ of the rv $X$ with respect to $\mathcal{G}$.

A~conditional right endpoint of the support, which can be viewed as a~conditional quantile of order~$\lambda=1$, is defined in an analogous way.

\begin{df}
    Suppose $X$ is an~rv and $\mathcal{G}$ is a~sigma-field with $\mathcal{G}\subseteq\mathcal{F}$. The conditional right endpoint of~the support of~$X$ given $\mathcal{G}$, denoted by $\gamma_1(X|\mathcal{G})$, is defined as an~extended rv $\underline{Q}_1$ with the following properties 
    \begin{description}
    \item[(i)]
    $\underline{Q}_1$ is $\mathcal{G}$-measurable,
    \item[(ii)]
    $\mathbb{P}(X\le \underline{Q}_1|\mathcal{G})=1$ a.s.,
    \item[(iii)]
    for any $\mathcal{G}$-measurable extended rv~$Q_1$ such that $\mathbb{P}(X\le Q_1|\mathcal{G})=1$ a.s., we have $Q_1\ge \underline{Q}_1$ a.s.
    \end{description}
\end{df}
Replacing $X$ by $-X$ in Theorem~\ref{tw.1}, we immediately obtain that for any rv~$X$ and any sigma-field $\mathcal{G}\subseteq\mathcal{F}$ there exists an extended  rv $\underline{Q}_1$ satisfying conditions (i)-(iii) of~the above definition. Moreover this rv is almost surely unique. It is also clear that for any rv~$X$ and any sigma-field $\mathcal{G}\subseteq\mathcal{F}$ we have
\begin{equation}\label{wzor4}
\gamma_1(X|\mathcal{G})=-\gamma_0(-X|\mathcal{G}).
\end{equation}
Relation~\eqref{wzor4} allows us to rewrite properties of~conditional left endpoints of~supports as that of conditional right endpoints of supports. Therefore in what follows we restrict our attention to properties of $\gamma_0(X|\mathcal{G})$.

\begin{tw}
\label{tw.2}
Let $X$ and $Y$ be rv's and $\mathcal{G}\subseteq\mathcal{F}$ be a~sigma-field. If~$Y$ is $\mathcal{G}$-measurable, then
\begin{description}
\item[(i)]
$\gamma_0(Y|\mathcal{G})=Y$ a.s.,
\item[(ii)]
$\gamma_0(X+Y|\mathcal{G})=\gamma_0(X|\mathcal{G})+Y$ a.s.,
\item[(iii)]
$\gamma_0(XY|\mathcal{G})=Y\gamma_0(X|\mathcal{G})$ a.s provided that $Y\ge 0$ a.s. or $\gamma_0(X|\mathcal{G})=X$ a.s.
\end{description}
\end{tw}
\begin{proof}
To prove (i), let $Q_0$ be any $\mathcal{G}$-measurable extended rv satisfying \linebreak $\mathbb{P}(Y\ge Q_0|\mathcal{G})=1$ a.s. Since $\mathbb{P}(Y\ge Q_0|\mathcal{G})=I(Y\ge Q_0)$, it follows that $I(Y\ge Q_0)=1$ a.s. and hence that $Y\ge Q_0$ a.s. Therefore $\gamma_0(Y|\mathcal{G})=Y$ a.s. by the definition of conditional left endpoint of support.

For (ii) observe that $\gamma_0(X|\mathcal{G})+Y$ is $\mathcal{G}$-measurable and 
\[
\mathbb{P}(X+Y\ge \gamma_0(X|\mathcal{G})+Y|\mathcal{G})= \mathbb{P}(X\ge \gamma_0(X|\mathcal{G})|\mathcal{G})=1\textrm{ a.s.}
\]
Next, let $Q_0$ be any $\mathcal{G}$-measurable extended rv satisfying $\mathbb{P}(X+Y\ge Q_0|\mathcal{G})=~1$ a.s. Then $\mathbb{P}(X\ge Q_0-Y|\mathcal{G})=1$ a.s. and since $Q_0-Y$ is $\mathcal{G}$-measurable, by the definition of $\gamma_0(X|\mathcal{G})$, we get $Q_0-Y\le \gamma_0(X|\mathcal{G})$ a.s., which means that $Q_0\le \gamma_0(X|\mathcal{G})+Y$ a.s. Thus $\gamma_0(X|\mathcal{G})+Y$ is indeed a~version of~$\gamma_0(X+Y|\mathcal{G})$.

To prove (iii) note that, on the event $[Y=0]$,
\[
\mathbb{P}(XY\ge Y\gamma_0(X|\mathcal{G})|\mathcal{G})=
\mathbb{P}(0\ge 0)=1\textrm{ a.s.},
\]
on the event $[Y>0]$,
\[
\mathbb{P}(XY\ge Y\gamma_0(X|\mathcal{G})|\mathcal{G})=
\mathbb{P}(X\ge\gamma_0(X|\mathcal{G})|\mathcal{G})=1\textrm{ a.s.}
\]
and on the event $[Y<0]$,
\begin{align*}
&\mathbb{P}(XY\ge Y\gamma_0(X|\mathcal{G})|\mathcal{G})=
\mathbb{P}(X\le\gamma_0(X|\mathcal{G})|\mathcal{G}) & \\
&=1-\mathbb{P}(X\ge\gamma_0(X|\mathcal{G})|\mathcal{G})+ \mathbb{P}(X=\gamma_0(X|\mathcal{G})|\mathcal{G})
=\mathbb{P}(X=\gamma_0(X|\mathcal{G})|\mathcal{G})\textrm{ a.s.} &
\end{align*}
It follows that if $\gamma_0(X|\mathcal{G})=X$ a.s. or $Y\ge 0$ a.s. then 
$$\mathbb{P}(XY\ge Y\gamma_0(X|\mathcal{G})|\mathcal{G})=1  \textrm{ a.s.}$$
Now, suppose $Q_0$ is any $\mathcal{G}$-measurable extended rv such that  \linebreak $\mathbb{P}(XY\ge Q_0|\mathcal{G})=1$ a.s. Then
$Q_0\le Y\gamma_0(X|\mathcal{G})$ a.s. provided that $Y\ge 0$ a.s. or  $\gamma_0(X|\mathcal{G})=X$  a.s. Indeed,  on the event $[Y>0]$,
$$\mathbb{P}(X\ge Q_0/Y|\mathcal{G})=1 \textrm{ a.s.},$$
which, by the  $\mathcal{G}$-measurability of $Q_0/Y$ and the definition of $\gamma_0(X|\mathcal{G})$, shows that $Q_0/Y\le \gamma_0(X|\mathcal{G})$  a.s.  and hence that $Q_0\le Y\gamma_0(X|\mathcal{G})$ a.s. Next,  on the event $[Y=0]$ we have 
$1=\mathbb{P}(0\ge Q_0|\mathcal{G})=I(Q_0\geq 0)$ a.s., which gives $Q_0\geq 0=Y\gamma_0(X|\mathcal{G})$ a.s. Finally, if  $X=\gamma_0(X|\mathcal{G})$  a.s. then we get
$$\mathbb{P}(\gamma_0(X|\mathcal{G})Y\ge Q_0|\mathcal{G})=1 \textrm{ a.s.}$$
and the  $\mathcal{G}$-measurability of $\gamma_0(X|\mathcal{G})Y$ and $Q_0$  impiles
$$I(\gamma_0(X|\mathcal{G})Y\geq Q_0)=1  \textrm{ a.s.}$$
Thus again $Y\gamma_0(X|\mathcal{G})\geq Q_0$ a.s. as claimed.
\end{proof}

\begin{tw}
\label{tw.3.4}
Let $a\in\mathbb{R}$, $X$ and $Y$ be rv's and $\mathcal{G}\subseteq\mathcal{F}$ be a~sigma-field. Then
\begin{description}
\item[(i)]
$\gamma_0(a|\mathcal{G})=a$ a.s.,
\item[(ii)]
$\gamma_0(aX|\mathcal{G})=a\gamma_0(X|\mathcal{G})$ a.s.
provided that $a\ge 0$,
\item[(iii)]
$X\ge Y$ a.s. implies
$\gamma_0(X|\mathcal{G})\ge\gamma_0(Y|\mathcal{G})$ a.s.,
\item[(iv)]
$X\ge a$ a.s. implies
$\gamma_0(X|\mathcal{G})\ge a$ a.s.,
\item[(v)]
if $X$ is independent of~$\mathcal{G}$, then
$\gamma_0(X|\mathcal{G})=\gamma_0(X|\mathcal{G}_0)=\gamma_0^{X}$ a.s., where $\mathcal{G}_0=\{\emptyset,\Omega\}$, the trivial sigma-field.
\end{description}
\end{tw}

\begin{proof}
Properties (i) and (ii) follow from Theorem~\ref{tw.2} (i) and (iii), respectively, by taking $Y=a$.

To prove (iii) note that, by the definition of~$\gamma_0(Y|\mathcal{G})$, we have \linebreak  $\mathbb{P}(Y\ge \gamma_0(Y|\mathcal{G})|\mathcal{G})=1$ a.s. By assumption, $X\ge Y$ a.s. This gives \linebreak $\mathbb{P}(X\ge \gamma_0(Y|\mathcal{G})|\mathcal{G})=1$ a.s. Hence $\gamma_0(Y|\mathcal{G})$ belongs to the family $\{Q_t,t\in T\}$ of $\mathcal{G}$-measurable extended rv's $Q_t$ such that $\mathbb{P}(X\ge Q_t|\mathcal{G})=1$ a.s. and consequently 
\[
\gamma_0(Y|\mathcal{G})\le \esssuptT Q_t= \gamma_0(X|\mathcal{G})\textrm{ a.s.}
\]

Property (iv) is a~direct consequence of~(i) and~(iii).

For (v) observe that (iv) together with the fact that $X\ge\gamma_0^{X}$ imply $\gamma_0(X|\mathcal{G})\ge \gamma_0^{X}$ a.s. Therefore the proof is completed by showing that
\begin{equation}\label{r1}
\gamma_0(X|\mathcal{G})\le \gamma_0^{X}\textrm{ a.s.}
\end{equation}
To see this, suppose, contrary to our claim, that~\eqref{r1} is not satisfied. Define $G=\{\omega: \gamma_0(X|\mathcal{G})(\omega)>\gamma_0^{X} \}$. If \eqref{r1} is not true then $\mathbb{P}(G)>0$. Next, let $G_n=\{\omega: \gamma_0(X|\mathcal{G})(\omega)>\gamma_0^{X} +\frac{1}{n} \}$, $n\ge 1$. Since $G=\bigcup_{n=1}^{\infty}G_n$ and $G_{n+1}\supseteq G_n$, we have $\mathbb{P}(G)=\lim_{n\to\infty} \mathbb{P}(G_n)$. The independence of~$X$ and~$\mathcal{G}$ gives
\begin{align}\label{r2}
&\lim_{n\to\infty} \mathbb{P}(\gamma_0(X|\mathcal{G})>\gamma_0^{X}+\tfrac{1}{n}|X\le \gamma_0^{X}+\tfrac{1}{n}) \nonumber \\ 
&= \lim_{n\to\infty} \mathbb{P}(\gamma_0(X|\mathcal{G})>\gamma_0^{X}+\tfrac{1}{n})=\lim_{n\to\infty} \mathbb{P}(G_n)=\mathbb{P}(G)>0.
\end{align}
On the other hand, for any $n\ge 1$,
\begin{align*}
\mathbb{P}(\gamma_0(X|\mathcal{G})>\gamma_0^{X}+\tfrac{1}{n}|X\le \gamma_0^{X}+\tfrac{1}{n})\le \mathbb{P}(\gamma_0^{X}+\tfrac{1}{n}>\gamma_0^{X}+\tfrac{1}{n}|X\le \gamma_0^{X}+\tfrac{1}{n})=0,
\end{align*}
because from $X\le \gamma_0^{X}+\frac{1}{n}$ and (iii) we get $\gamma_0(X|\mathcal{G})\le\gamma_0^{X}+\frac{1}{n}$, by (i).
This clearly forces
\begin{align*}
\lim_{n\to\infty} \mathbb{P}(\gamma_0(X|\mathcal{G})>\gamma_0^{X}+\tfrac{1}{n}|X\le \gamma_0^{X}+\tfrac{1}{n})=0,
\end{align*}
contrary to~\eqref{r2}.
\end{proof}

\begin{tw}
\label{stalyKwantyl}
Let $X$ be an rv and $\mathcal{G}\subseteq\mathcal{F}$ be a~sigma-field. If $\gamma_0(X|\mathcal{G})$ is almost surely constant, then $\gamma_0(X|\mathcal{G})=\gamma_0^{X}$ a.s.
\end{tw}
\begin{proof}
By assumption,  $\gamma_0(X|\mathcal{G})=c$ a.s. for some $c\in[-\infty,\infty)$. By the definition of $\gamma_0(X|\mathcal{G})$,
\begin{align*}
\mathbb{P}(X\geq c+\delta)=\mathbb{P}(X\geq \gamma_0(X|\mathcal{G})+\delta)=\mathbb{E}(\mathbb{P}(X\geq \gamma_0(X|\mathcal{G})+\delta|\mathcal{G})) & \\
=\left\{  
\begin{array}{lcl}
\mathbb{E}(1)=1 &  \textrm{ if  } &  \delta=0, \\
\mathbb{E}(P_{\delta})=p_{\delta} &  \textrm{ if  } &  \delta>0,  
\end{array} \right.&
\end{align*}
where the rv $P_{\delta}\in[0,1]$ is such that $P_{\delta}<1$ with positive probability, which implies $p_{\delta}\in[0,1)$. Hence 
$$\gamma_0^X=\inf\{x\in\R: \mathbb{P}(X\geq x)<1\}=c,$$
which proves the theorem.
\end{proof}

\begin{tw}
\label{s3wlgraniczne}
Let $\mathcal{G}$ be a~sigma-field with $\mathcal{G}\subseteq\mathcal{F}$ and $X,X_1,X_2,\ldots$ be rv's.  If $X_n\downarrow X$ a.s., then
\begin{equation}
\label{s3nw1}
\gamma_0(X_n|\mathcal{G})\xrightarrow{a.s.} \gamma_0(X|\mathcal{G})
\end{equation}
\end{tw}
\begin{proof}
First note that, since by assumption
\[
X_{n+1}\le X_n\textrm{ a.s. for all }n\ge 1,
\]
Theorem \ref{tw.3.4}~(iii) gives
\[
\gamma_0(X_{n+1}|\mathcal{G})\le\gamma_0(X_n|\mathcal{G})\textrm{ a.s. for all }n\ge 1.
\]
This means that the sequence $(\gamma_0(X_n|\mathcal{G}), n\ge 1)$ is nonincreasing a.s., which implies
\begin{equation}\label{s3nw2}
\lim_{n\to\infty}\gamma_0(X_n|\mathcal{G})\textrm{ exists (possibly infinite) a.s.}
\end{equation}
Moreover, the $\mathcal{G}$-measurability of $\gamma_0(X_n|\mathcal{G})$, $n\ge 1$, shows that
\[
\Omega_0=\{\omega\in\Omega\colon \lim_{n\to\infty}\gamma_0(X_n|\mathcal{G})\textrm{ exists (possibly infinite)}\}\in\mathcal{G},
\]
and by~\eqref{s3nw2} we have
\begin{equation}\label{s3nw3}
\mathbb{P}(\Omega_0)=1.
\end{equation}

Let~$\overline{Q}_0=\lim_{n\to\infty}\gamma_0(X_n|\mathcal{G})\cdot I(\Omega_0)$. We will prove that $\overline{Q}_0$ is a~version of~$\gamma_0(X|\mathcal{G})$ and hence~\eqref{s3nw1} holds, by~\eqref{s3nw3}. First note that~$\overline{Q}_0$ is $\mathcal{G}$-measurable as a~limit of~$\mathcal{G}$-measurable extended rv's $\gamma_0(X_n|\mathcal{G})\cdot I(\Omega_0)$, $n\ge 1$. Next, by Lemma~\ref{l1}, $\mathbb{P}(X\ge \overline{Q}_0|\mathcal{G})=1$ a.s. Therefore it remains to show that if~$Q_0$ is a~$\mathcal{G}$-measurable extended rv such that $\mathbb{P}(X\ge Q_0|\mathcal{G})=1$ a.s., then $Q_0\le\overline{Q}_0$ a.s. To do this, observe that by assumption
\[
X_n\ge X\textrm{ a.s. for all }n\ge 1,
\]
which, by the monotonicity property of conditional expectations, gives
\[
1=\mathbb{P}(X\ge Q_0|\mathcal{G})\le\mathbb{P}(X_n\ge Q_0|\mathcal{G})\textrm{ a.s.}
\]
This clearly forces $\mathbb{P}(X_n\ge Q_0|\mathcal{G})=1$ a.s. We conclude from the definition of~$\gamma_0(X_n|\mathcal{G})$ that $Q_0\le \gamma_0(X_n|\mathcal{G})$ a.s. for all $n\ge 1$, hence that $Q_0\le\lim_{n\to\infty}\gamma_0(X_n|\mathcal{G})$ a.s. and finally that $Q_0\le\overline{Q}_0$ a.s. The proof is complete.
\end{proof}

It is worth pointing out that in Theorem~\ref{s3wlgraniczne} the assumption that $X_n\downarrow X$ a.s. cannot be replaced by $X_n\uparrow X$, even if we additionally require that $X$ is bounded. This is shown in the following example.

\begin{ex}
Let $(\Omega,\mathcal{F},\mathbb{P})=([0,1],\mathcal{B}([0,1]),Leb)$, where $\mathcal{B}([0,1])$ denotes the Borel sigma-field of subsets of~$[0,1]$ and $Leb$ stands for Lebesgue measure. On this probability space define $X_n=I([\frac{1}{n},1])$, $n\ge 1$, and $X=1$. Then $X_n\uparrow X$ a.s. and $X$ is bounded but, by Theorem~\ref{tw.3.4}~(v),
\[
\gamma_0(X_n|\mathcal{G}_0)=0\to 0\neq \gamma_0(X|\mathcal{G}_0)=1\textrm{ a.s.}
\]
with $\mathcal{G}_0=\{\emptyset,[0,1]\}$.
\end{ex}


\section{The strong ergodic theorem}

The aim of this section is to provide a complete generalization of~Theorem~\ref{l3} by quiting the ergodicity assumption. To state and prove this result we need not only the new concepts of conditional left and right endpoints of a~support but also some terminology and facts from the ergodic theory. 

By $(\mathbb{R}^{\mathbb{N}},\mathcal{B}(\mathbb{R}^{\mathbb{N}}),\mathbb{Q})$ we denote a~probability triple, where $\mathbb{R}^{\mathbb{N}}$ is the set of sequences of real numbers $(x_1,x_2,\ldots)$, $\mathcal{B}(\mathbb{R}^{\mathbb{N}})$ stands for the Borel sigma-field of subsets of $\mathbb{R}^{\mathbb{N}}$ and $\mathbb{Q}$ is a~stationary probability measure on the pair $(\mathbb{R}^{\mathbb{N}},\mathcal{B}(\mathbb{R}^{\mathbb{N}}))$.

A~set $B\in \mathcal{B}(\mathbb{R}^{\mathbb{N}})$ is called
\begin{itemize}
    \item 
    invariant if $B=T^{-1}B$,
    \item
    almost invariant for $\mathbb{Q}$ if 
    \[
    \mathbb{Q}((B\setminus T^{-1}B)\cup(T^{-1}B\setminus B))=0,
    \]
\end{itemize}
where 
\begin{equation}\label{transfT}
T^{-1}B=\{(x_1,x_2,\ldots)\in\mathbb{R}^{\mathbb{N}}:\ (x_2,x_3,\ldots)\in B \}.
\end{equation}
The class of all invariant events is denoted by $\tilde{\mathcal{I}}$, while the class of all almost invariant events for $\mathbb{Q}$ is denoted by $\mathcal{I}_{\mathbb{Q}}$. The following properties of $\mathcal{I}_{\mathbb{Q}}$ and $\tilde{\mathcal{I}}$ are well known; see, for example, Durrett (\cite{dur2010}, Chapter 6) and Shiryaev (\cite{shir1996}, Chapter V).

\begin{lem}\label{l2}
\begin{description}
\item[(i)]
    $\tilde{\mathcal{I}}$ and $\mathcal{I}_{\mathbb{Q}}$ are sigma-fields.
\item[(ii)]
    An rv $X$ on $(\mathbb{R}^{\mathbb{N}},\mathcal{B}(\mathbb{R}^{\mathbb{N}}),\mathbb{Q})$ is 
$\tilde{\mathcal{I}}$-measurable (or $\mathcal{I}_{\mathbb{Q}}$-measurable) iff
  \[
    X((x_1,x_2,\ldots))=X((x_2,x_3,\ldots))\textrm{ for all }   (x_1,x_2,\ldots)\in\mathbb{R}^{\mathbb{N}}
    \]
    \[
   (\textrm{or }  X((x_1,x_2,\ldots))=X((x_2,x_3,\ldots))\textrm{ for }\mathbb{Q}\textrm{-almost every }(x_1,x_2,\ldots)\in\mathbb{R}^{\mathbb{N}}.
    \]
\item[(iii)]
    If $B$ is almost invariant, there exists an invariant set $C$ such that 
    \[
    \mathbb{Q}((B\setminus C)\cup(C\setminus B))=0.
    \]
\end{description}
\end{lem}

Now we are ready to formulate and prove the first version of the strong ergodic theorem for extreme and intermediate order statistics. 
\begin{tw}\label{tw.3}
Let $Y$ be an rv on a probability space $(\mathbb{R}^{\mathbb{N}},\mathcal{B}(\mathbb{R}^{\mathbb{N}}),\mathbb{Q})$, where the probability
measure $\mathbb{Q}$ is stationary. Suppose that the sequence of rv's $(Y_n, n \ge 1)$ is defined by
\begin{equation}\label{r6}
Y_i ((x_1, x_2,\ldots)) = Y ((x_i , x_{i+1},\ldots))\textrm{ for }(x_1, x_2,\ldots) \in \mathbb{R}^{\mathbb{N}}\textrm{ and }i \ge 1.
\end{equation}
If $(k_n, n \ge 1)$ is a~sequence of  integers satisfying~\eqref{warK}
then
\begin{equation}\label{teza4.1}
Y_{k_n:n}\xrightarrow{\mathbb{Q}-a.s.}\gamma_0(Y|\mathcal{I}_{\mathbb{Q}})\  (\gamma_1(Y|\mathcal{I}_{\mathbb{Q}}))\textrm{ according as }\lambda=0\ (\lambda=1).
\end{equation}
\end{tw}

\begin{proof}
We may assume that $\lambda=0$, because the case $\lambda=1$ is an~immediate consequence of the former. 
If we prove that 
\begin{equation}\label{8}
\limsup_{n\to\infty}Y_{k_n:n}\le\gamma_0(Y|\mathcal{I}_{\mathbb{Q}})\textrm{ a.s.}
\end{equation}
and
\begin{equation}\label{inf}
\liminf_{n\to\infty}Y_{k_n:n}\ge\gamma_0(Y|\mathcal{I}_{\mathbb{Q}})\textrm{ a.s.},
\end{equation}
the assertion follows. Let us first show~\eqref{8}. For $m\ge 1$, define rv's $D_m$ by
$$D_m((x_1,x_2,\ldots))=\left\{
\begin{array}{l}
-m   \quad \textrm{ if  }  \quad  \gamma_0(Y|\mathcal{I}_{\mathbb{Q}})((x_1,x_2,\ldots))=-\infty, \\
\gamma_0(Y|\mathcal{I}_{\mathbb{Q}})((x_1,x_2,\ldots))+\frac{1}{m}   \qquad \textrm{ otherwise. } 
\end{array}
\right.
$$
 Then $D_m$ is $\mathcal{I}_{\mathbb{Q}}$-measurable, which by part~(ii) of Lemma~\ref{l2} gives
\begin{align}
\label{pD}
&  D_m((x_i,x_{i+1},\ldots))  \\
& =D_m((x_1,x_2,\ldots))\textrm{ for any $i\geq 1$ and }\mathbb{Q}\textrm{-almost every } (x_1,x_2,\ldots)\in\mathbb{R}^{\mathbb{N}}. \nonumber
\end{align}
Fix $m\ge 1$. As in the proof of Theorem~\ref{l3} we get
\begin{align}\label{9}
\mathbb{Q}(Y_{k_n:n}\le D_m\textrm{ for all large }n)
&=\mathbb{Q}\Big(\frac{\sum_{i=1}^{n}I(Y_i\le D_m)}{n}-\frac{k_n}{n}\ge 0\textrm{ for all large }n\Big).
\end{align}
Set $Z=I(Y\le D_m)$ and
\[
Z_i ((x_1, x_2,\ldots)) = Z ((x_i , x_{i+1},\ldots))\textrm{ for }(x_1, x_2,\ldots) \in \mathbb{R}^{\mathbb{N}}\textrm{ and }i \ge 1.
\]
Then,  for $\mathbb{Q}$-almost every  $(x_1,x_2,\ldots)\in\mathbb{R}^{\mathbb{N}}$,
\begin{align}\label{A}
    Z_i ((x_1, x_2,\ldots)) &= Z ((x_i , x_{i+1},\ldots))= I(Y\le D_m)((x_i , x_{i+1},\ldots))\nonumber\\
    &= I(Y((x_i , x_{i+1},\ldots))\le D_m((x_i , x_{i+1},\ldots)))\nonumber\\
    &= I(Y_i((x_1 , x_{2},\ldots))\le D_m((x_1 , x_{2},\ldots))),
\end{align}
where the last equality is a~consequence of~\eqref{pD}.

Since $Z$ is an~rv on~$(\mathbb{R}^{\mathbb{N}},\mathcal{B}(\mathbb{R}^{\mathbb{N}}),\mathbb{Q})$ and $\mathbb{E}_{\mathbb{Q}}(|Z|)<1$, the classic strong ergodic theorem (see, for example, Durrett~\cite{dur2010}, p.333) gives
\begin{equation}\label{nowe0}
\frac{1}{n}\sum_{i=1}^n Z_i\xrightarrow{\mathbb{Q}-a.s.}\mathbb{E}_{\mathbb{Q}}(Z|\mathcal{I}_{\mathbb{Q}}),
\end{equation}
which by~\eqref{A} means that
\begin{equation}\label{9A}
\frac{1}{n}\sum_{i=1}^n I(Y_i\le D_m)\xrightarrow{\mathbb{Q}-a.s.} \mathbb{E}_{\mathbb{Q}}(I(Y\le D_m)|\mathcal{I}_{\mathbb{Q}})= \mathbb{Q}(Y\le D_m|\mathcal{I}_{\mathbb{Q}}).
\end{equation}
We claim that 
\begin{equation}\label{r8}
\mathbb{Q}(Y\le D_m|\mathcal{I}_{\mathbb{Q}})((x_1, x_2,\ldots))>0\textrm{ for }\mathbb{Q}\textrm{-almost every } (x_1,x_2,\ldots)\in\mathbb{R}^{\mathbb{N}}.
\end{equation}
To prove this, suppose, contrary to our claim, that $\mathbb{Q}(G)> 0$, where 
$$G:=\{(x_1, x_2,\ldots)\in\mathbb{R}^{\mathbb{N}}:\ \mathbb{Q}(Y\le D_m|\mathcal{I}_{\mathbb{Q}})((x_1 , x_{2},\ldots))= 0\}.$$ 
Let 
$$G_1=G\cap\{(x_1, x_2,\ldots)\in\mathbb{R}^{\mathbb{N}}:\ \gamma_0(Y|\mathcal{I}_{\mathbb{Q}})((x_1 , x_{2},\ldots))= -\infty\}$$
 and 
$$G_2=G\cap\{(x_1, x_2,\ldots)\in\mathbb{R}^{\mathbb{N}}:\ \gamma_0(Y|\mathcal{I}_{\mathbb{Q}})((x_1 , x_{2},\ldots))> -\infty\}.$$
 Then $\mathbb{Q}(G_1)> 0$ or $\mathbb{Q}(G_2)> 0$. On $G_1$ we have 
\[
0=\mathbb{Q}(Y\le D_m|\mathcal{I}_{\mathbb{Q}})=\mathbb{Q}(Y\le -m|\mathcal{I}_{\mathbb{Q}})\textrm{ a.s.},
\]
which implies
\begin{equation}\label{B}
\mathbb{Q}(Y\ge -m|\mathcal{I}_{\mathbb{Q}})=1\textrm{ a.s. on }G_1.
\end{equation}
Similarly on $G_2$
\[
0=\mathbb{Q}(Y\le D_m|\mathcal{I}_{\mathbb{Q}})=\mathbb{Q}(Y\le \gamma_0(Y|\mathcal{I}_{\mathbb{Q}})+\tfrac{1}{m}|\mathcal{I}_{\mathbb{Q}})\textrm{ a.s.},
\]
and consequently
\begin{equation}\label{C}
\mathbb{Q}(Y\ge \gamma_0(Y|\mathcal{I}_{\mathbb{Q}})+\tfrac{1}{m}|\mathcal{I}_{\mathbb{Q}})=1\textrm{ a.s. on }G_2.
\end{equation}
If $\mathbb{Q}(G_1)> 0$, defining $Q_1((x_1,x_2,\ldots))=-m$ for $(x_1,x_2,\ldots)\in G_1$ and $Q_1((x_1,x_2,\ldots))=\gamma_0(Y|\mathcal{I}_{\mathbb{Q}})((x_1,x_2,\ldots))$ otherwise, we would get that the following three conditions were satisfied. 
\begin{enumerate}
    \item 
    $Q_1$ is $\mathcal{I}_{\mathbb{Q}}$-measurable.
    \item
    $\mathbb{Q}(Y\ge Q_1|\mathcal{I}_{\mathbb{Q}})=1$ a.s. by~\eqref{B} and the definition of~$\gamma_0(Y|\mathcal{I}_{\mathbb{Q}})$.
    \item
    It is not true that $Q_1\le\gamma_0(Y|\mathcal{I}_{\mathbb{Q}})$ a.s. since $Q_1>\gamma_0(Y|\mathcal{I}_{\mathbb{Q}})$ on $G_1$.
\end{enumerate}
This contradicts the definition of $\gamma_0(Y|\mathcal{I}_{\mathbb{Q}})$. 

If in turn $\mathbb{Q}(G_2)> 0$, we take 
$$Q_2((x_1,x_2,\ldots))=
\left\{
\begin{array}{ll}
\gamma_0(Y|\mathcal{I}_{\mathbb{Q}})((x_1,x_2,\ldots))+\frac{1}{m} & \textrm{ if  }   (x_1,x_2,\ldots)\in G_2, \\
\gamma_0(Y|\mathcal{I}_{\mathbb{Q}})((x_1,x_2,\ldots)) & \textrm{ otherwise}.
\end{array}
\right.
$$
Using similar reasoning to the above, we would again contradict the definition of $\gamma_0(Y|\mathcal{I}_{\mathbb{Q}})$. 

Thus~\eqref{r8} is proved. Combining~\eqref{9A} with~\eqref{r8} we see that, as \linebreak $n\to\infty$,
\[
\frac{1}{n}\sum_{i=1}^n I(Y_i\le D_m)-\frac{k_n}{n}\to  \mathbb{Q}(Y\le D_m|\mathcal{I}_{\mathbb{Q}})>0 \quad \mathbb{Q}\textrm{-a.s.},
\]
which by~\eqref{9} leads to
\[
\mathbb{Q}(Y_{k_n:n}\le D_m\textrm{ for all large }n)
=1.
\]
This clearly forces
\[
\limsup_{n\to\infty}Y_{k_n:n}\le D_m \quad \mathbb{Q}\textrm{-a.s.}
\]
Since $m\ge 1$ was taken to be arbitrary, letting $m\to\infty$ and using the countability of the set of positive integers, we get 
\[
\limsup_{n\to\infty}Y_{k_n:n}\le \gamma_0(Y|\mathcal{I}_{\mathbb{Q}})\quad \mathbb{Q}\textrm{-a.s.},
\]
which establishes~\eqref{8}.

It remains to prove~\eqref{inf}. For this purpose, observe that
\[
1=\mathbb{E}_{\mathbb{Q}}(\mathbb{Q}(Y\ge\gamma_0(Y|\mathcal{I}_{\mathbb{Q}})|\mathcal{I}_{\mathbb{Q}}))
=\mathbb{Q}(Y\ge\gamma_0(Y|\mathcal{I}_{\mathbb{Q}})),
\]
where the first equality is a~consequence of~the definition of~$\gamma_0(Y|\mathcal{I}_{\mathbb{Q}})$. This gives
\begin{equation}\label{D}
Y_i\ge\gamma_0(Y|\mathcal{I}_{\mathbb{Q}})\ \mathbb{Q}\textrm{-a.s. for each }i\ge 1. 
\end{equation}
Indeed, we have, for any $i\ge 1$,
\begin{align*}
&\mathbb{Q}(Y_i\ge\gamma_0(Y|\mathcal{I}_{\mathbb{Q}})) \\
&=
\mathbb{Q}(\{(x_1,x_2,\ldots):\ Y_i((x_1,x_2,\ldots))\ge\gamma_0(Y|\mathcal{I}_{\mathbb{Q}})((x_1,x_2,\ldots))\})\\
&= \mathbb{Q}(\{(x_1,x_2,\ldots):\ Y((x_i,x_{i+1},\ldots))\ge\gamma_0(Y|\mathcal{I}_{\mathbb{Q}})((x_i,x_{i+1},\ldots))\})\\
&= \mathbb{Q}(T^{-(i-1)}(\{(x_1,x_2,\ldots):\ Y((x_1,x_2,\ldots))\ge\gamma_0(Y|\mathcal{I}_{\mathbb{Q}})((x_1,x_2,\ldots))\}))\\
&= \mathbb{Q}(\{(x_1,x_2,\ldots):\ Y((x_1,x_2,\ldots))\ge\gamma_0(Y|\mathcal{I}_{\mathbb{Q}})((x_1,x_2,\ldots))\})\\ &= \mathbb{Q}(Y\ge\gamma_0(Y|\mathcal{I}_{\mathbb{Q}}))=1\textrm{ a.s.},
\end{align*}
where the transformation $T^{-1}$ is defined in~\eqref{transfT} and the fourth equality follows from the stationarity of the measure $\mathbb{Q}$. Note that~\eqref{D} implies, for all $n\ge 1$,
\[
Y_{k_n:n}\ge\gamma_0(Y|\mathcal{I}_{\mathbb{Q}})\quad \mathbb{Q}\textrm{-a.s.},
\]
which gives~\eqref{inf}. The proof is complete.
\end{proof}

\begin{rem}
Since~\eqref{nowe0} is still true if we replace $\mathcal{I}_{\mathbb{Q}}$ by $\tilde{\mathcal{I}}$ (see, for example, Grimmet and Stirzaker \cite{grim2004}, Chapter~9), and we have a version of Lemma  \ref{l2} (ii) for $\mathcal{\tilde{I}}$-measurable rv's,  the~conclusion of~Theorem~\ref{tw.3} can as well have the following form
\begin{equation}\label{tezabis}
Y_{k_n:n}\xrightarrow{\mathbb{Q}-a.s.}\gamma_0(Y|\mathcal{\tilde{I}})\ (\gamma_1(Y|\mathcal{\tilde{I}}))\textrm{ according as }\lambda=0\ (\lambda=1).
\end{equation}
\end{rem}

Theorem~\ref{tw.3} deals with the almost sure limit of extreme and intermediate order statistics arising from the specific random sequence $(Y_n, n\ge 1)$ defined on a~probability space $(\mathbb{R}^{\mathbb{N}},\mathcal{B}(\mathbb{R}^{\mathbb{N}}),\mathbb{Q})$ by \eqref{r6}.

Our goal now is to reformulate this result in terms of~any strictly stationary sequence of rv's $(X_n,n\ge 1)$ existing in any probability space $(\Omega,\mathcal{F},\mathbb{P})$. To do this, for arbitrary such a~sequence $\mathbb{X}=(X_n;n\ge 1)$ we define a~stationary measure $\mathbb{Q}$ on the pair
$(\mathbb{R}^{\mathbb{N}},\mathcal{B}(\mathbb{R}^{\mathbb{N}}))$ by
\begin{equation}\label{12}
\mathbb{Q}(A)=\mathbb{P}(\mathbb{X}\in A)\textrm{ for all }A\in\mathcal{B}(\mathbb{R}^{\mathbb{N}}).
\end{equation}
Next, on the triple $(\mathbb{R}^{\mathbb{N}},\mathcal{B}(\mathbb{R}^{\mathbb{N}}),\mathbb{Q})$ we introduce an~rv $Y\colon\mathbb{R}^{\mathbb{N}}\to\mathbb{R}$ by
\begin{equation}\label{12A}
Y((x_1,x_2,\ldots))=x_1\textrm{ for }(x_1,x_2,\ldots)\in\mathbb{R}^{\mathbb{N}}
\end{equation}
and a~sequence of rv's $\mathbb{Y}=(Y_n,n\ge 1)$ by
\begin{equation}\label{12B}
Y_i((x_1,x_2,\ldots))=x_i\textrm{ for }(x_1,x_2,\ldots)\in\mathbb{R}^{\mathbb{N}},\ i\ge 1.
\end{equation}
Then \eqref{r6} holds and Theorem~\ref{tw.3} shows that, for any sequence of integers satisfying~\eqref{warK}, \eqref{teza4.1} is true. Since $\mathbb{X}$ and $\mathbb{Y}$ have the same distributions, the $\mathbb{Q}$-almost sure convergence of the sequence $(Y_{k_n:n},n\ge 1)$ to $\gamma_0(Y|\mathcal{I}_{\mathbb{Q}})$ ($\gamma_1(Y|\mathcal{I}_{\mathbb{Q}})$) entails the $\mathbb{P}$-almost sure convergence of $(X_{k_n:n},n\ge 1)$ to an~rv~$W$ such that $W\stackrel{d}{=}\gamma_0(Y|\mathcal{I}_{\mathbb{Q}})$ ($\gamma_1(Y|\mathcal{I}_{\mathbb{Q}})$). To describe the structure of~$W$, we need some more facts from the ergodic theory.

Recall that a~set $A\in\mathcal{F}$ is called invariant with respect to the sequence $\mathbb{X}=(X_n,n\ge 1)$ defined on~the probability space $(\Omega,\mathcal{F},\mathbb{P})$ if there exists a~set $B\in\mathcal{B}(\mathbb{R}^{\mathbb{N}})$ such that
\begin{equation}\label{warAinvar}
A=\{\omega\in\Omega:\ (X_i(\omega),X_{i+1}(\omega),\ldots)\in B \}\textrm{ for any }i\ge 1.
\end{equation}
The collection of all such invariant sets is denoted by $\mathcal{I}^{\mathbb{X}}$.
\begin{lem}\label{InvSeq}
Let $\mathbb{X}=(X_n,n\ge 1)$ be a~strictly stationary sequence on $(\Omega,\mathcal{F},\mathbb{P})$.
\begin{description}
 \item[(i)]
    $\mathcal{I}^{\mathbb{X}}$ is a~sigma-field.
 \item[(ii)] $\mathcal{I}^{\mathbb{X}}\subseteq \mathcal{T}^{\mathbb{X}}$, where $\mathcal{T}^{\mathbb{X}}$ denotes the tail sigma-field generated by the sequence $\mathbb{X}$.
  \item[(iii)]
    A~set $A\in\mathcal{F}$ is invariant with respect to~$\mathbb{X}$ if and only if there exists a~set $B\in\mathcal{I}_{\mathbb{Q}}$ satisfying~\eqref{warAinvar}.
 \item[(iv)]
    If an~rv $X$ on~$(\Omega,\mathcal{F},\mathbb{P})$ is $\mathcal{I}^{\mathbb{X}}$-measurable then there exists an $\mathcal{I}_{\mathbb{Q}}$-measurable rv~$Q$ on~$(\mathbb{R}^{\mathbb{N}},\mathcal{B}(\mathbb{R}^{\mathbb{N}}),\mathbb{Q})$ such that $X=Q((X_1,X_2,\ldots))$.
\end{description}
\end{lem}

\begin{proof}
Part (i) is known; see, for example, Shiryaev (\cite{shir1996}, Chapter~V). 

For (ii) observe that, by the definition of $\mathcal{I}^{\mathbb{X}}$,  we have, for any $A\in \mathcal{I}^{\mathbb{X}}$,
 $$A\in\sigma(X_i,X_{i+1},\ldots) \textrm{  for all } i\geq 1.$$  This gives $A\in \mathcal{T}^{\mathbb{X}}$ and so $\mathcal{I}^{\mathbb{X}}\subseteq \mathcal{T}^{\mathbb{X}}$.

To show part~(iii), note that if $B\in \mathcal{B}(\mathbb{R}^{\mathbb{N}})$ satisfies~\eqref{warAinvar} then
\begin{equation}\label{nowy1}
\{\omega\in\Omega:\ (X_1(\omega),X_2(\omega),\ldots)\in B\}
= \{\omega\in\Omega:\ (X_2(\omega),X_3(\omega),\ldots)\in B\}.
\end{equation}
From~\eqref{12}
\begin{align*}
&\mathbb{Q}((B\setminus T^{-1}B)\cup(T^{-1}B\setminus B)) \\
&=\mathbb{P}(\{\omega: (X_1(\omega),X_2(\omega),\ldots)\in B\textrm{ and }(X_2(\omega),X_3(\omega),\ldots)\notin B\})\\
&+ \mathbb{P}(\{\omega: (X_2(\omega),X_3(\omega),\ldots)\in B\textrm{ and }(X_1(\omega),X_2(\omega),\ldots)\notin B\})\\
&=2\mathbb{P}(\emptyset)=0,
\end{align*}
by~\eqref{nowy1}. This means that $B\in\mathcal{I}_{\mathbb{Q}}$.

To prove part~(iv), one can use the Monotone-Class Theorem and part~(iii) of~Lemma~\eqref{InvSeq}, repeating reasoning given in Williams (\cite{wil1991}, Chapter~A3). 
\end{proof}

The following theorem asserts that the rv~$W$ such that $W=\lim_{n\to\infty}X_{k_n:n}$ $\mathbb{P}$-a.s. can be taken equal to~$\gamma_0(X_1|\mathcal{I}^{\mathbb{X}})$ ($\gamma_1(X_1|\mathcal{I}^{\mathbb{X}})$).

\begin{tw}\label{ergodyczneNowe}
Let $\mathbb{X}=(X_n,n\ge 1)$ be a~strictly stationary sequence and $(k_n,n\ge 1)$ be a~sequence of  integers satisfying~\eqref{warK}. Then
\[
X_{k_n:n}\xrightarrow{\mathbb{P}-a.s.}\gamma_0(X_1|\mathcal{I}^{\mathbb{X}})\ (\gamma_1(X_1|\mathcal{I}^{\mathbb{X}}))\textrm{ according as }\lambda=0\ (\lambda=1).
\]
\end{tw}

\begin{proof}
As in the proof of~Theorem~\ref{tw.3}, we can restrict ourselves to the case~$\lambda=0$. From the previous discussion, we already know that $\lim_{n\to\infty} X_{k_n:n}$ exists $\mathbb{P}$-a.s. (possibly infinite). For definitness on the set (of probability $\mathbb{P}$ zero) of $\omega\in\Omega$ such that $\lim_{n\to\infty}X_{k_n:n}(\omega)$ does not exist, let us set, for example, $\lim_{n\to\infty}X_{k_n:n}(\omega)=-\infty$. The proof is completed by showing that the following three conditions hold.
\begin{enumerate}
\item
$\lim_{n\to\infty}X_{k_n:n}$ is $\mathcal{I}^{\mathbb{X}}$-measurable.
\item
$\mathbb{P}(X_1\ge \lim_{n\to\infty}X_{k_n:n}|\mathcal{I}^{\mathbb{X}})=1$ $\mathbb{P}$-a.s.
\item
If $\tilde{W}$ is an $\mathcal{I}^{\mathbb{X}}$-measurable rv such that
\begin{equation}\label{nowy1A}
\mathbb{P}(X_1\ge \tilde{W}|\mathcal{I}^{\mathbb{X}})=1\quad  \mathbb{P}\textrm{-a.s.}
\end{equation}
\end{enumerate}
then $\tilde{W}\le\lim_{n\to\infty}X_{k_n:n}$ $\mathbb{P}$-a.s.

Condition~1 means that, for all $A\in\mathcal{B}(\mathbb{R})$,
$$\{\omega\in\Omega:\ \lim_{n\to\infty}X_{k_n:n}(\omega)\in A\}\in \mathcal{I}^{\mathbb{X}},$$
which, by the definition of $\mathcal{I}^{\mathbb{X}}$, is equivalent to the following requirement:
\begin{align*}
\textrm{for all } A\in\mathcal{B}(\mathbb{R}) \textrm{ there exists } B\in \mathcal{B}(\mathbb{R}^{\mathbb{N}})  \textrm{ such that for all  } n\ge 1 \\
\{\omega\in\Omega:\ \lim_{n\to\infty}X_{k_n:n}(\omega)\in A\}=\{\omega\in\Omega:\ (X_n(\omega),X_{n+1}(\omega),\ldots)\in B\}.
\end{align*}
One can take $B=\{(x_1,x_2,\ldots)\in\mathbb{R}^{\mathbb{N}}:\ \lim_{n\to\infty} x_{k_n:n}\in A\}$, where $\lim_{n\to\infty}x_{k_n:n}$ is defined to equal $-\infty$ if this limit does not exist. Indeed, by~\eqref{tezabis}, $B\in\tilde{\mathcal{I}}$, which ensures that $B\in \mathcal{B}(\mathbb{R}^{\mathbb{N}})$ and $\{\omega\in\Omega:\ (X_1(\omega),X_{2}(\omega),\ldots)\in B\}=\{\omega\in\Omega:\ (X_n(\omega),X_{n+1}(\omega),\ldots)\in B\}$ for all $n\ge 1$.

 Showing condition~2 amounts to proving that
\[
\mathbb{E}_{\mathbb{P}}(I(X_1\ge \lim_{n\to\infty}X_{k_n:n}) I(A))=\mathbb{E}_{\mathbb{P}}(1\cdot I(A))\quad\textrm{ for all }A\in\mathcal{I}^{\mathbb{X}}.
\]
This is equivalent to
\begin{equation}\label{nowy2}
\mathbb{P}(\{\omega\in\Omega\colon X_1(\omega)\ge\lim_{n\to\infty}X_{k_n:n}(\omega)\}\cap A)=\mathbb{P}(A) \quad\textrm{ for all }A\in\mathcal{I}^{\mathbb{X}}.
\end{equation}
By part~(iii) of Lemma~\ref{InvSeq}, for any $A\in\mathcal{I}^{\mathbb{X}}$ there exists $B\in\mathcal{I}_{\mathbb{Q}}$ such that~\eqref{warAinvar} holds. Therefore to prove~\eqref{nowy2} it suffices to show that
\begin{equation}\label{nowy3}
\mathbb{Q}(\{(x_1,x_2,\ldots)\in\mathbb{R}^{\mathbb{N}}: x_1\ge \lim_{n\to\infty}x_{k_n:n}\}\cap B)=\mathbb{Q}(B)\textrm{ for all  }B\in\mathcal{I}_{\mathbb{Q}}.
\end{equation}
To see this, note that by~\eqref{teza4.1},
\[
\mathbb{E}_{\mathbb{Q}}(I(Y\ge \lim_{n\to\infty}Y_{k_n:n}) I(B))=\mathbb{E}_{\mathbb{Q}}(1\cdot I(B))\quad\textrm{ for all }B\in\mathcal{I}_{\mathbb{Q}},
\]
which is equivalent to~\eqref{nowy3}. The proof of condition~2 is completed.

For condition~3, assume that~\eqref{nowy1A} holds for some $\mathcal{I}^{\mathbb{X}}$-measurable rv $\tilde{W}$. Then, by part~(iv) of~Lemma~\ref{InvSeq}, there exists an~$\mathcal{I}_{\mathbb{Q}}$-measurable rv~$Q$ such that $\tilde{W}=Q((X_1,X_2,\ldots))$. Hence \eqref{nowy1A} 
can be rewritten as
\[
\mathbb{P}(\{\omega\in\Omega\colon X_1(\omega)\ge Q((X_1(\omega), X_2(\omega),\ldots))\}\cap A)=\mathbb{P}(A)\textrm{ for all }A\in\mathcal{I}^{\mathbb{X}},
\]
which implies
\[
\mathbb{Q}(\{(x_1,x_2,\ldots)\in\mathbb{R}^{\mathbb{N}}: x_1\ge Q((x_1,x_2,\ldots))\}\cap B)=\mathbb{Q}(B)\textrm{ for all  }B\in\tilde{\mathcal{I}},
\]
because for any $B\in\tilde{\mathcal{I}}$, $A=\{\omega\in\Omega\colon (X_1(\omega),X_2(\omega),\ldots)\in B\}\in\mathcal{I}^{\mathbb{X}}$. By part~(iii) of~Lemma~\ref{l2}, we also have
\[
\mathbb{Q}(\{(x_1,x_2,\ldots)\in\mathbb{R}^{\mathbb{N}}: x_1\ge Q((x_1,x_2,\ldots))\}\cap B)=\mathbb{Q}(B)\textrm{ for all  }B\in\mathcal{I}_{\mathbb{Q}},
\]
which is equivalent to
\[
\mathbb{Q}(\{Y\ge Q\}\cap B)=\mathbb{Q}(B)\textrm{ for all  }B\in\mathcal{I}_{\mathbb{Q}},
\]
that is to
\[
\mathbb{Q}(Y\ge Q|\mathcal{I}_{\mathbb{Q}})=1\quad \mathbb{Q}\textrm{-a.s.}
\]
Since Q is $\mathcal{I}_{\mathbb{Q}}$-measurable, by the definition of~$\gamma_0(Y|\mathcal{I}_{\mathbb{Q}})$, we get
\[
Q\le\gamma_0(Y|\mathcal{I}_{\mathbb{Q}})\quad \mathbb{Q}\textrm{-a.s.}
\]
Theorem~\ref{tw.3} implies that
\[
\mathbb{Q}(Q\le\lim_{n\to\infty} Y_{k_n:n})=1
\]
and hence that
\[
\mathbb{P}(\omega\in\Omega\colon Q((X_1(\omega), X_2(\omega),\ldots))\le\lim_{n\to\infty}X_{k_n:n}(\omega))=1,
\]
which is the desired conclusion.
\end{proof}

\section{Examples}
We will apply results of previous sections to some families of strictly stationary sequences of rv's. In particular, we will show that Example \ref{ex1} and Theorem \ref{l3} are special cases of Theorem \ref{ergodyczneNowe}.

\subsection{Sequences of identical rv's}
Let $X_n=X$ for all $n\ge 1$, where $X$ is some rv. Then $\mathbb{X}=(X_n,n\ge 1)$ is strictly stationary. Moreover
\begin{equation}\label{s5w1}
\mathcal{I}^{\mathbb{X}}=\sigma(X).
\end{equation}
Indeed, in this case $\mathcal{T}^{\mathbb{X}}=\sigma(X)$ so part~(ii) of Lemma~\ref{InvSeq} gives $\mathcal{I}^{\mathbb{X}}\subseteq\sigma(X)$. To show that $\sigma(X)\subseteq\mathcal{I}^{\mathbb{X}}$ it suffices to observe that
\begin{align*}
    \sigma(X)&=\{ \{\omega\in\Omega\colon X(\omega)\in A\}\colon A\in\mathcal{B}(\mathbb{R})\}\\
    &\subseteq 
    \{ \{\omega\in\Omega\colon (X(\omega),X(\omega),\ldots)\in B\}\colon B\in\mathcal{B}(\mathbb{R}^{\mathbb{N}})\} =\mathcal{I}^{\mathbb{X}}.
\end{align*}
By~\eqref{s5w1} Theorems~\ref{ergodyczneNowe} and~\ref{tw.2}~(i) immediately give
\[
X_{k_n:n}\xrightarrow{\mathbb{P}\textrm{-a.s.}}\gamma_0(X_1|\mathcal{I}^{\mathbb{X}}) =\gamma_0(X|\sigma(X))=X\ (\gamma_1(X_1|\mathcal{I}^{\mathbb{X}})=X)
\]
according as $\lambda=0$ ($\lambda=1$)
for any sequence $(k_n,n\ge 1)$ of integers satisfying~\eqref{warK}. Note that the above conclusion agrees with that of~Example 2.1.

\subsection{Strictly stationary and ergodic processes}
Let $\mathbb{X}=(X_n,n\ge 1)$ be a~strictly stationary and ergodic sequence of rv's. Ergodicity means that the measure of any set $A\in\mathcal{I}^{\mathbb{X}}$ is either 0 or 1; see, for example Shiryaev (\cite{shir1996}, p.413). Consequently any $\mathcal{I}^{\mathbb{X}}$-measurable extended rv is $\mathbb{P}$-almost surely constant. Indeed, let $Z$ be $\mathcal{I}^{\mathbb{X}}$-measurable. Then, for any $a\in\mathbb{R}$,
\[
\{\omega\in\Omega\colon Z(\omega)\le a\}\in\mathcal{I}^{\mathbb{X}}\textrm{ and so }\mathbb{P}(Z\le a)=0\textrm{ or 1}.
\]
By taking $a_0=\sup\{a\in\mathbb{R}\colon\mathbb{P}(Z\le a)=0\}$, we get
\begin{align*}
\mathbb{P}(Z\le a)
=\left\{  
\begin{array}{lcl}
0 &  \textrm{ if } &  a<a_0 \\
1 &  \textrm{ if } &  a>a_0
\end{array}, \right.
\end{align*}
which clearly forces $\mathbb{P}(Z=a_0)=1$ as required.

In particular $\gamma_0(X_1|\mathcal{I}^{\mathbb{X}})$ is $\mathbb{P}$-almost surely constant as an $\mathcal{I}^{\mathbb{X}}$-measurable extended rv. Hence by Theorem~\ref{stalyKwantyl} we have $\gamma_0(X_1|\mathcal{I}^{\mathbb{X}})=\gamma_0^{X_1}$ $\mathbb{P}$-a.s. Using the same arguments we show that also $\gamma_1(X_1|\mathcal{I}^{\mathbb{X}})=\gamma_1^{X_1}$ $\mathbb{P}$-a.s. Now Theorem~\ref{ergodyczneNowe} gives, for any sequence $(k_n,n\ge 1)$ of positive integers satisfying~\eqref{warK},
\[
X_{k_n:n}\to\gamma_0(X_1|\mathcal{I}^{\mathbb{X}}) =\gamma_0^{X_1}\ (\gamma_1(X_1|\mathcal{I}^{\mathbb{X}})=\gamma_1^{X_1})\quad \mathbb{P}\textrm{-a.s.}
\]
according as $\lambda=0$ ($\lambda=1$). Thus we have deduced Theorem~\ref{l3} as a~special case of Theorem~\ref{ergodyczneNowe}.

\subsection{Random shift and scaling of strictly stationary and ergodic processes}
Using results of Sections~5.1 and~5.2 we can describe the almost sure limiting behaviour of extreme and intermediate order statistics corresponding to the following sequences of rv's:
\[
\mathbb{R}_n=(R_n,n\ge 1),\quad \mathbb{S}_n=(S_n,n\ge 1),
\]
where $R_n=X_n+U$, $S_n=V\cdot X_n$, $n\ge 1$, $(X_n,n\ge 1)$ is a~strictly stationary and ergodic process, $U$ is an~rv and $V$ is a~non-negative rv. Indeed, for every $1\le k\le n$, $R_{k:n}=X_{k:n}+U$ and $S_{k:n}=V\cdot X_{k:n}$. Therefore, for any sequence $(k_n, n\ge 1)$ of positive integers satisfying~\eqref{warK}, we get, as $n\to\infty$,
\[
R_{k_n:n}=X_{k_n:n}+U\xrightarrow{\mathbb{P}-a.s.}\gamma_0^{X_1}+U\quad (\gamma_1^{X_1}+U)
\]
and
\[
S_{k_n:n}=V\cdot X_{k_n:n}\xrightarrow{\mathbb{P}-a.s.}V\cdot\gamma_0^{X_1}\quad (V\cdot\gamma_1^{X_1}),
\]
according as $\lambda=0$ ($\lambda=1$).

\appendix
\section{Appendix}

For the convenience of the reader we recall here the definition of essential supremum and its existence theorem. This material is taken from Chow, Robbins and Siegmund (\cite{chow1971}, Chapter~1).
\begin{df}
We say that a~rv $Y$ is the essential supremum of a~family of rv's $\{X_t,t\in T\}$ and write $Y=\textrm{ess} \sup_{t\in T} X_t$ if
\begin{description}
 \item[(i)]
 $\mathbb{P}(Y\ge X_t)=1$ for every $t\in T$;
 \item[(ii)]
 if $\tilde{Y}$ is any rv such that $\mathbb{P}(\tilde{Y}\ge X_t)=1$ for every $t\in T$, then $\tilde{Y}\ge Y$ a.s.
\end{description}
\end{df}

\begin{tw}
For any family of rv's $\{X_t,t\in T\}$, $Y=\textrm{ess}\sup_{t\in T}X_t$ exists, and for some countable subset $C$ of~$T$ we have
\[
Y=\sup_{t\in C} X_t.
\]
\end{tw}



\begin{thebibliography}{99}


\bibitem{brad2007}  \textsc{Bradley, R. T.} (2007).
\emph{Introduction to Strong Mixing Conditions},
Vol. 1, Kendrick Press, Heber City, Utah.

\bibitem{Ch92} \textsc{Chanda, K.C.}   (1992).  
Bahadur-Kiefer representation properties of intermediate order statistics.
\textit{Statist. Probab. Lett.} \textbf{14} 175--178.

\bibitem{Ch85} \textsc{Cheng, S.} (1985).
On limiting distributions of order statistics with variable ranks from stationary sequences.
\textit{Ann. Probab.} \textbf{13} 1326--1340.

\bibitem{chow1971} \textsc{Chow, Y. S., Robbins, H.} and \textsc{Siegmund, D.} (1971).
\emph{Great Expectations: The Theory of Optimal Stopping},
Houghton Mifflin, New York.

\bibitem{dem2012} \textsc{Dembi\'nska, A.} (2012). 
Limit theorems for proportions of observations falling into
random regions determined by order statistics.
\textit{Aust. N. Z. J. Stat.} \textbf{54} 199--210.

\bibitem{dem2014} \textsc{Dembi\'nska, A.} (2014). 
Asymptotic behaviour of central order statistics from stationary processes.
\textit{Stochastic Process. Appl.} \textbf{124} 348--372.

\bibitem{dur2010} \textsc{Durrett, R.} (2010).
\emph{Probability: Theory and Examples},
4th ed. Cambridge U. Press.

\bibitem{embre1997}  \textsc{Embrechts, P., Kl\"uppelberg, C.} and \textsc{Mikosch, T.} (1997).
\emph{Modelling Extremal Events for Insurance and Finance},
Springer-Verlag, Berlin.

\bibitem{F05}  \textsc{Fasen, V.} (2005). Extremes of regularly varying L\'evy-driven mixed moving average processes.
\textit{Adv. in Appl. Probab.} \textbf{37} 993--1014. 

\bibitem{grim2004} \textsc{Grimmet, G. R.} and \textsc{ Stirzaker, D. R.} (2004).
\emph{Probability and Random Processes},
Oxford University Press, New York.

\bibitem{K84} \textsc{Klass, M. J.} (1984).
 The minimal growth rate of partial maxima. 
\textit{Ann. Probab.} \textbf{12} 380-389.

\bibitem{K85} \textsc{Klass, M. J.} (1985).
The Robbins-Siegmund series criterion for partial maxima. 
\textit{Ann. Probab.} \textbf{13} 1369-1370.

\bibitem{LLR83} \textsc{Leadbetter, M. R.,  Lindgren, G.} and  \textsc{ Rootz\'en, H.}  (1983).
\emph{Extremes and Related Properties of Random Sequences and Processes},
Springer-Verlag, New York.

\bibitem{LR88} \textsc{Leadbetter, M. R.} and \textsc{Rootz\'en, H.} (1988).
Extremal theory for stochastic processes.
\textit{Ann. Probab.} \textbf{16} 431--478.

\bibitem{MS00} \textsc{Mikosch, T.} and \textsc{St\u{a}ric\u{a}, C.}  (2000). Limit theory for the sample autocorrelations and
extremes of a GARCH(1, 1) process. \textit{Ann. Statist.} \textbf{28} 1427--1451. 

\bibitem{shir1996} \textsc{Shiryaev, A. N.} (1996). 
\emph{Probability}, 2nd ed.
 Springer, New York.

\bibitem{smirnov1952} \textsc{Smirnov, N. V.} (1952). 
Limit distributions for the terms of a variational series.
\textit{Amer. Math. Soc. Transl. Ser. 1}  \textbf{11} 
82--143. Original published in 1949.

\bibitem{tomkins} \textsc{Tomkins, R. J.} (1975).
On conditional medians. 
\textit{Ann. Probab.} \textbf{3} 375--379.

\bibitem{W97} \textsc{Wang, H.} (1997).
Generalized zero-one laws for large-order statistics.
\textit{Bernoulli} \textbf{3} 429-444.

\bibitem{W80} \textsc{Watts, V.} (1980).
The almost sure representation of intermediate order statistics.
\textit{Z. Wahrsch. Verw. Gebiete} \textbf{54} 281--285.

\bibitem{WRL82} \textsc{Watts, V.,   Rootz\'en, H.} and \textsc{Leadbetter, M. R.}  (1982).
On limiting distributions of intermediate order statistics from stationary sequences.
\textit{Ann. Probab.} \textbf{10} 653--662.

\bibitem{wil1991} \textsc{Williams, D.} (1991). 
\emph{Probability with Martingales},
Cambridge University Press.

\end{thebibliography}
\end{document}